\documentclass{amsart}
 \usepackage{amsmath}
 \usepackage{epsfig} 
 \usepackage{graphics} 

\theoremstyle{plain}
\newtheorem{thm}{Theorem}[section]

\newtheorem{prop}[thm]{Proposition}
\newtheorem{clly}[thm]{Corollary}
\newtheorem{lemma}[thm]{Lemma}

\newtheorem{defi}[thm]{Definition}

\newtheorem{ex}[thm]{Example}
\newtheorem{maintheorem}{Theorem}

\newcommand{\spec}{\operatorname{Spc}}

\newcommand{\cl}{\operatorname{Cl}}

\newcommand{\diff}{\operatorname{Diff}}
\newcommand{\per}{\operatorname{Per}}
\newcommand{\sink}{\operatorname{Sink}}
\newcommand{\sou}{\operatorname{Source}}

\newcommand{\sad}{\operatorname{Saddle}}

\newcommand{\tr}{\operatorname{Tr}}

\title[Characterizing finite sets of nonwandering points]
      {Characterizing finite sets of nonwandering points}

\author[C. A. Morales]{C. A. Morales}
\address{Instituto de Mat\'ematica, UFRJ,
P. O. Box 68530, 21945-970 Rio de Janeiro, Brazil.}


\email{morales@impa.br}

\thanks{Partially supported by CNPq, FAPERJ and Pronex Dyn. Systems. from Brazil.}

\begin{document}




\begin{abstract}
We characterize finite sets $S$ of nonwandering points
for generic diffeomorphisms $f$ as those which are
{\em uniformly bounded}, i.e.,
there is an uniform bound for small perturbations of the derivative of $f$ along the points in $S$
up to suitable iterates. We use this result
to give a $C^1$ generic characterization of the Morse-Smale diffeomorphisms related to
the weak Palis conjecture \cite{c}.
Furthermore, we obtain another proof of the result by Liao and Pliss
about the finiteness of sinks and sources for star diffeomorphisms \cite{l}, \cite{Pl}.
\end{abstract}

\maketitle

\section{Introduction}

\noindent
The study of subsets of nonwandering points for smooth diffeomorphisms is an interesting topic
in the hyperbolic theory of dynamical systems.
To support this assertion we can mention the {\em Smale's conjecture} that the
Axiom A diffeomorphisms constitute an open and dense set
in the space of $C^1$ diffeomorphisms of any closed surface (\cite{sm} p. 779).
Indeed, the {\em Ma\~n\'e's dichotomy} \cite{M} reduces it to prove
that every $C^1$ generic surface diffeomorphism has {\em finitely many sinks and sources}
(actually it suffices to rule out the existence of infinitely many periodic points with nonreal eigenvalues
for $C^1$ generic surface diffeomorphisms \cite{m}).
These works trigger the search for necessary and sufficient conditions for the finiteness
of a given subset of nonwandering points for diffeomorphisms $f$ on closed manifolds
and, in this paper, we focus on those $f$
exhibiting finitely many periodic points of period $n$, $\forall n\in\mathbb{N}^+$.
The condition we are interested in relies on
the definition of {\em uniformly bounded set} $S$ with respect to some map
$n: S\to \mathbb{N}^+$
which, roughly speaking, means that
there is an uniform bound for small perturbations of the derivative of $f$ along points $p\in S$
up to $n(p)$-iterates.
We shall prove that a set $S$ of nonwandering points is finite as soon as it is uniformly bounded
with respect to some $n: S\to \mathbb{N}^+$
with {\em finite preimages} (i.e. $n^{-1}(k)$
is finite for all $k\in\mathbb{N}^+$).
As an application of this result
we will obtain a $C^1$ generic characterization of the Morse-Smale diffeomorphisms related to
the {\em weak Palis's conjecture} \cite{c}. Afterward
we obtain another proof of the result, by Liao and Pliss \cite{l}, \cite{Pl}, that
all star diffeomorphisms on closed manifolds have finitely many sinks and sources.
Let us state our results in a precise way.

Hereafter $M$ is a {\em closed manifold}, i.e., a compact connected boundaryless manifold of dimension $\dim(M)\geq 2$.
Denote by $\|\cdot\|$ some Riemannian metric on $M$.
The space of $C^1$ diffeomorphisms equipped with the standard
$C^1$ topology is denoted by $\diff^1(M)$.
It turns out that $\diff^1(M)$ is a Baire space, and so,
residual (i.e. countable intersection of open and dense) subsets are
dense. We say that a certain property holds for $C^1$ generic diffeomorphisms
if it does in a residual subset of $\diff^1(M)$.
Given $f\in \diff^1(M)$ we say that $p\in M$ is a
{\em wandering point} if $U\cap\left(\cup_{n\in\mathbb{N}^+}f^n(U)\right)=\emptyset$ for some neighborhood
$U$ of $p$. Otherwise we call it {\em nonwandering point}.
The set of nonwandering points is the {\em nonwandering set} denoted by
$\Omega(f)$.
A point $p$ is {\em periodic} if there is a positive integer
$n$ such that $f^n(p)=p$.
The minimal of such integers is the period of $p$ denoted by
$n_p$ (or $n_{p,f}$ to emphasize $f$). We denote by $\per(f)$ the set of periodic points.
Clearly $\per(f)\subset \Omega(f)$ although the inclusion may be proper.
The eigenvalues of a periodic point $p$ are those of the linear mapping
$Df^{n_p}(p):T_pM\to T_pM$. We say that $p$ is  {\em hyperbolic} if its eigenvalues have modulus different from $1$,
a {\em sink} if all its eigenvalues have modulus less than $1$,
a {\em source} if it is a sink for $f^{-1}$ and a {\em saddle} if it is neither a sink nor a source.
By the invariant manifold theory \cite{hps} every hyperbolic periodic point $p$ is equipped with
a pair of invariant manifolds, the stable and unstable manifolds, tangent
at $p$ to the eigenspace associated to the eigenvalues of modulus less than and bigger than $1$ respectively.
We say that $f$ is {\em Morse-Smale} if $\Omega(f)$
consists of finitely many hyperbolic periodic points all of whose invariant manifolds are in general position.
The product of linear mappings $A,B$ will be denoted by $AB$
(or $\prod_{i=1}^nA_i$ when finitely many maps $A_1,\cdots, A_n$).
Given two sets $O$ and $Q$ we say that a map $n: O\to Q$ has
{\em finite preimages} if $n^{-1}(k)$ is finite for every $k\in Q$.
The following is the main definition of this work.

\begin{defi}
\label{def1}
A (nonnecessarily invariant) set $S\subset M$ is
{\em uniformly bounded with respect to a map $n: S\to \mathbb{N}^+$} if
there are positive constants $\epsilon,K$ such that
\begin{equation}
\label{uniformly}
\left\|\displaystyle\prod_{i=0}^{n(p)-1}L_i\right\|\leq K,
\end{equation}
for all $p\in S$ and all sequence
of linear isomorphisms $L_i: T_{f^i(p)}M\to T_{f^{i+1}(p)}M$ with
$\|L_i-Df(f^i(p))\|\leq \epsilon$ for $0\leq i\leq n(p)-1$.
\end{defi}

Related to this definition we obtain the following example.

\begin{ex}
Every set $S$ is uniformly bounded with respect to
any {\em bounded} map $n:S\to \mathbb{N}^+$.
\end{ex}

\begin{proof}
Indeed, we obtain (\ref{uniformly})
by taking any upper bound $n_0$ of $n$,
any $\epsilon>0$ and $K=(\|f\|+\epsilon)^{n_0}$
where
$
\|f\|=\sup_{p\in M}\|Df(p)\|.
$
\end{proof}

This example suggests that extra hypotheses on
$n: S\to \mathbb{N}^+$ are needed in order to obtain interesting results.
The one we shall consider here is that of having {\em finite preimages},
i.e., $n^{-1}(k)$ is finite for every $k\in \mathbb{N}^+$.
At first glance we can check easily that a bounded map $n: S\to \mathbb{N}^+$ has
finite preimages if and only if $S$ is finite.
From this it follows that every finite set $S$ is uniformly bounded with respect to some
map $n: S\to \mathbb{N}^+$ with finite primages.
This elementary observation makes us to ask if, conversely,
every set $S$ which is uniformly bounded with respect to some
map with finite preimages $n:S\to \mathbb{N}^+$ is finite.
Nevertheless, the answer is negative by the following counterexample.

\begin{ex}
If $f\in \diff^1(M)$ exhibits a sink $p$ of period $1$
for which $\|Df(p)\|<1$, then $f$ also exhibits an infinite set $S$
which is uniformly bounded with respect to some map with finite preimages $n: S\to \mathbb{N}^+$.
\end{ex}

\begin{proof}
Take $\delta>0$ with $\|Df(p)\|+\delta<1$ and a neighborhood $U$ of $p$ such that
$\|Df(x)\|\leq \frac{\delta}{2}+\|Df(p)\|$ for all $x\in U$.
Since $p$ has period $1$ there is a neighborhood $W\subset U$ of $p$ such that
$f^k(x)\in U$ for all $x\in W$.
We shall prove that any infinite sequence $S=\{x_k:k\in\mathbb{N}^+\}$ in $W$ is uniformly bounded
with respect to $n: S\to \mathbb{N}^+$, $n(x_k)=k$
(clearly $n^{-1}(k)=\{x_k\}$ for all $k$ so $n$ has finite preimages).
Define $\epsilon=\frac{\delta}{2}$, fix $k\in \mathbb{N}^+$ and let
$L_i: T_{f^i(x_k)}M\to T_{f^{i+1}(x_k)}M$ be any sequence of linear isomorphisms satisfying
$\|L_i-Df(f^i(x_k))\|\leq \epsilon$ for all $0\leq i\leq k-1$.
Then,
$
\|L_i\|\leq\epsilon+\|Df(f^i(x_k))\|<\delta+\|Df(p)\|
$
so
$$
\left\|\displaystyle\prod_{i=0}^{k-1}L_i\right\|
\leq \displaystyle\prod_{i=0}^{k-1}\|L_i\|
\leq(\delta+\|Df(p)\|)^k<1
$$
and we are done.
\end{proof}

Now observe that, in this counterexample, every point in $S$ is
wandering.
This observation motivates our main result below.

\begin{maintheorem}
 \label{thA}
Every set $S$ of nonwandering points which is uniformly bounded
with respect to some map with finite preimages $n:S\to \mathbb{N}^+$ is finite.
\end{maintheorem}

The proof is based on methods introduced by
Ma\~n\'e \cite{M} and Pliss \cite{Pl}.

In the sequel we present some applications of this theorem.
The first one is the following characterization.

\begin{clly}
 \label{cofin}
If $f\in\diff^1(M)$, then $S\subset\Omega(f)$ is finite if and only if
there are a neighborhood $\mathcal{U}$ of $f$ and a map with finite preimages $n: S\to\mathbb{N}^+$ such that
$$
\sup_{(h,p)\in \mathcal{U}\times S}\|Dh^{n(p)}(p)\|<\infty.
$$
\end{clly}

The second application is needed to justify the supremum in Corollary \ref{M-S}.
For every $h\in \diff^1(M)$ we denote by
$\per^*(h)$ the set of periodic points of $h$ which are not sinks,
i.e., with at least one eigenvalue of modulus bigger than or equal to $1$.

\begin{clly}
 \label{necessary1}
$\{f\in \diff^1(M):\per^*(f)\neq\emptyset\}$ is open and dense in $\diff^1(M)$.
\end{clly}

The next application is a condition for existence of finitely many periodic points.

\begin{clly}
 \label{M-S}
Let $f$ be a diffeomorphism of a closed manifold with
finitely many periodic points of period $n$, $\forall n\in\mathbb{N}^+$.
If there is a neighborhood $\mathcal{U}$ of $f$ such that
\begin{equation}
\label{ms1}
\sup_{(h,q)\in\mathcal{U}\times\per^*(h)} \|Dh^{n_{q,h}}(q)\|<\infty,
\end{equation}
then $f$ has finitely many periodic points.
\end{clly}

This corollary may be used to characterize Morse-Smale diffeomorphisms on closed manifolds.
Indeed, since all such diffeomorphisms are structural stable \cite{pt}
we have that all such diffeomorphisms come equipped with a neighborhood $\mathcal{U}$ satisfying
(\ref{ms1}).
It is then natural to ask if, conversely, every diffeomorphism exhibiting a neighborhood
$\mathcal{U}$ satisfying (\ref{ms1}) is Morse-Smale.
Although the answer is negative (counterexamples can be easily obtained as in \cite{hy}, \cite{s})
Corollary \ref{M-S} together with the Kupka-Smale and general density theorems
\cite{pt}, \cite{p} supply positive answer for $C^1$ generic diffeomorphisms.
More precisely, we have the following result.

\begin{clly}
\label{caca}
Every $C^1$ generic diffeomorphism of a closed manifold exhibiting a
neighborhood $\mathcal{U}$ satisfying (\ref{ms1}) is Morse-Smale.
\end{clly}

Notice that corollaries \ref{necessary1} and \ref{caca} (but not \ref{M-S})
can be obtained also from the {\em weak Palis's conjecture} recently solved in \cite{c}.
Indeed, in the case of Corollary \ref{caca}, the Birkhoff-Smale theorem \cite{pt} implies that
every diffeomorphism for which there is a neighborhood $\mathcal{U}$
satisfying (\ref{ms1}) is far from ones with homoclinic points, whereas,
by the weak Palis's conjecture, every $C^1$ generic diffeomorphism
far from homoclinic points is Morse-Smale.
However our approach is simpler since it does not use the weak Palis's conjecture.

A third application is as follows.
We say that $f\in\diff^1(M)$ is a {\em star diffeomorphism} if
there is a neighborhood $\mathcal{U}\subset \diff^1(M)$ of $f$
such that every periodic point
of every diffeomorphism in $\mathcal{U}$ is hyperbolic.
Using Theorem \ref{thA} we obtain another proof of the following result due to
Liao and Pliss \cite{l}, \cite{Pl}.

\begin{clly}
\label{li-pli}
Every star diffeomorphism of a closed manifold has finitely many sinks and sources.
\end{clly}

It is worth to note the role played by this corollary in both the solution of the
$C^1$ stability conjecture (\cite{M}, \cite{M1}) and in the characterization of
star diffeomorphisms on closed manifolds (as the Axiom A ones \cite{a}, \cite{haa}).
The proof of this corollary given here will not use these outstanding results.

The final application is the following condition for existence of finitely many sinks.
Given $f\in\diff^1(M)$ we denote by $\sink(f)$, $\sou(f)$ and $\sad(f)$
the set of sinks, sources and saddles of $f$.

\begin{clly}
\label{finite-sinks}
A $C^1$ generic diffeomorphism $f$ on a closed manifold satisfying
$$
\cl(\sad(f)\cup\sou(f))\cap \cl(\sink(f))=\emptyset
$$
has finitely many sinks.
\end{clly}

Apparently this corollary follows also from \cite{Pl}.

\section{Proof of Theorem \ref{thA}}

\noindent
We star with an useful characterization of the sinks of a given diffeomorphisms $f$.

To motivate let us observe that every sink of $f$ is a point $p\in\Omega(f)$ for which
there are constants $K>0$, $0<\gamma<1$ and $m_0\in\mathbb{N}^+$ satisfying
\begin{equation}
\label{sisi}
\displaystyle\prod_{j=0}^{l-1}\| Df^{m_0}(f^{m_0j}(p))\|\leq K\gamma^l,
\quad\quad\forall l\in\mathbb{N}^+.
\end{equation}
Indeed, we obtain (\ref{sisi}) by taking $n$ large such that $\gamma=\|Df^{n_pn}(p)\|<1$,
$m_0=n_pn$ and $K=1$.

We shall prove later in Proposition \ref{sink} that, conversely,
every nonwandering point $p$ exhibiting these constants is a sink.
The elementary lemma below reduces it to prove that $p$ is periodic.

\begin{lemma}
 \label{short1}
If $f\in\diff^1(M)$ then every $p\in \per(f)$ for which there are $K>0$, $0<\gamma<1$ and $m_0\in\mathbb{N}^+$ satisfying
(\ref{sisi}) is a sink of $f$.
\end{lemma}

\begin{proof}
Define
$$
\lambda=\gamma^{\frac{1}{m_0}}\quad\mbox{( thus }0<\lambda<1)\quad\mbox{ and }
\quad
K_0=K\sup_{0\leq r\leq m_0-1   }\left(\frac{\|f\|}{\lambda}\right)^r.
$$
If $l\geq m_0$ then
$l=m_0n+r$ for some integers $n\geq 1$ and $0\leq r\leq m_0-1$ so
$$
\|Df^l(p)\|
=\|D(f^r\circ (f^{m_0})^n)(p)\|
=\|Df^r(f^{m_0n}(p))\cdot D(f^{m_0})^n(p)\|\leq
$$
$$
\|Df^r(f^{m_0n}(p))\|\cdot \|D(f^{m_0})^n(p)\|\overset{(\ref{sisi})}{\leq}
\|f\|^rK\gamma^n
=
$$
$$
\|f\|^rK\gamma^{-\frac{r}{m_0}}\left(\gamma^{\frac{1}{m_0}}\right)^l
=\left(\frac{\|f\|}{\lambda}\right)^rK\left(\gamma^{\frac{1}{m_0}}\right)^l
$$
proving
\begin{equation}
 \label{alternative2}
\|Df^l(p)\|\leq K_0\lambda^l,
\quad\quad
\forall l\geq m_0.
\end{equation}
Now suppose that
$p$ has an eigenvalue $\lambda$ with modulus $|\lambda|\geq1$.
Then $Df^{n_pn}(p)v=\lambda^nv$ and so
$\|Df^{n_pn}(p)v\|\geq\|v\|$ for all $n\in\mathbb{N}^+$.
But (\ref{alternative2}) implies
$\|Df^{n_pn}(p)v\|\leq K_0\lambda^{n_pn}\|v\|$ for $n$ large yielding
$\|Df^{n_pn}(p)v\|\to 0$ as $n\to\infty$ thus $v=0$ which is absurd. Therefore,
$p\in\sink(f)$.
\end{proof}

Hereafter we denote by $d(\cdot,\cdot)$ the distance induced by the Riemannian metric of $M$ and by
$B(\cdot,\cdot)$ the corresponding open ball operation.
The following lemma seems to be well known (see for instance p. 213 in \cite{Pl} and p. 1978 in \cite{PS'}).
Its proof is included for the sake of completeness.

\begin{lemma}
 \label{blabla}
If $f\in\diff^1(M)$, $p\in M$, $K>0$, $0<\gamma<1$ and $m_0\in\mathbb{N}^+$ satisfy
(\ref{sisi}), then there are $\rho>0$, $K_0>0$ and $0<\lambda<1$ satisfying
\begin{equation}
\label{sisi1}
d(f^l(x),f^l(y))\leq K_0\lambda^l d(x,y),
\quad\quad\forall x,y\in B(p,\rho), \forall l\in \mathbb{N}.
\end{equation}
\end{lemma}

\begin{proof}
Put $g=f^{m_0}$,
take $\beta>0$ such that $(1+\beta)\gamma<1$
and define $\lambda_0=(1+\beta)\gamma$ (thus $0<\lambda_0<1$).
Fix $\epsilon_1>0$ such that
$$
\|Dg(x)\|\leq (1+\beta)\|Dg(y)\|
\quad\mbox{ whenever }
\quad
d(x,y)\leq \epsilon_1.
$$
It follows that if $n\in\mathbb{N}^+$ and $x,y\in M$ satisfy
$d(g^j(x),g^j(y))\leq \epsilon_1$,
$0\leq j\leq n-1$, then
\begin{equation}
 \label{estrela1}
\displaystyle\prod_{j=0}^{n-1}\|Dg(g^j(x))\|\leq (1+\beta)^n\displaystyle\prod_{j=0}^{n-1}
\|Dg(g^j(y))\|.
\end{equation}
Define $\hat{K}=\max\{K,1\}$ and take
$$
0<\rho<\frac{\epsilon_1}{\hat{K}}.
$$
We claim that
\begin{equation}
 \label{chega1}
d(g^l(x),g^l(y))\leq \hat{K}\lambda_0^ld(x,y),
\quad\forall x,y\in B(p,\rho), l\in\mathbb{N}.
\end{equation}
Indeed, it suffices to prove that the following assertion holds $\forall n\in\mathbb{N}^+$:
$$
d(g^l(x),g^l(y))\leq \hat{K}\lambda_0^l d(x,y),
\quad\quad\forall x,y\in B(p,\rho), \forall 0\leq l\leq n-1.
$$
Since $\hat{K}\geq 1$ we have the assertion for $n=1$.
Now suppose that the assertion holds for some $n\geq 2$.
Then, for every $z\in B(p,\rho)$ one has
$d(g^j(z),g^j(p))\leq \hat{K}\lambda_0^jd(z,p)<\hat{K}\lambda_0^j\rho\leq\epsilon_1$
for all $0\leq j\leq n-1$ by the choice of $\rho$.
Thus,
$$
\|Dg^n(z)\|\leq\displaystyle\prod_{j=0}^{n-1}\|Dg(g^j(z))\|\overset{(\ref{estrela1})}{\leq}
(1+\beta)^n\displaystyle\prod_{j=0}^{n-1}\| Dg(g^j(p))\|
\overset{(\ref{sisi})}{\leq}K((1+\beta)\gamma)^n=K\lambda_0^n
$$
proving
$$
\sup_{z\in B(p,\rho)}\|Dg^n(z)\|\leq K\lambda_0^n.
$$
Henceforth, for all $x,y\in B(p,\rho)$ one has
$$
d(g^n(x),g^n(y))\leq \left(\sup_{z\in B(p,\rho)}\|Dg^n(z)\|\right)\cdot d(x,y)
\leq K\lambda_0^n d(x,y)\leq \hat{K}\lambda_0^nd(x,y)
$$
and the assertion follows by induction.
This proves (\ref{chega1}).

Replacing $g=f^{m_0}$ in (\ref{chega1}) we have
\begin{equation}
 \label{paja1}
d(f^{m_0l}(x),f^{m_0l}(y))\leq \hat{K}\lambda_0^l d(x,y),
\quad\quad\forall x,y\in B(p,\rho), \forall l\in \mathbb{N}.
\end{equation}

Let $K_2$ be a Lipschitz constant of $f$, i.e.,
$d(f(x),f(y))\leq K_2d(x,y)$ for all $x,y\in M$.
Define
$$
\lambda=\lambda_0^{\frac{1}{m_0}}\mbox{ (thus }0<\lambda<1)
\quad\quad
\mbox{ and }
\quad\quad
K_0=\sup_{0\leq r\leq m_0-1}\max\left\{\frac{K_2^r}{\lambda^r},K_2^r\hat{K}\lambda_0^{-\frac{r}{m_0}}\right\}.
$$
Fix $x,y\in B(p,\rho)$.
It follows from the choices above that
$$
d(f^n(x),f^n(y))\leq K_0\lambda^nd(x,y),
\quad\quad
\forall 0\leq n\leq m_0-1.
$$
For $n\geq m_0$ we have $n=m_0l+r$ for some integers $l\geq 1$ and $0\leq r\leq m_0$.
Then,
$$
d(f^n(x),f^n(y))=d(f^r(f^{m_0l}(x)),f^r(f^{m_0l}(y)))
\leq K_2^rd(f^{m_0l}(x),f^{m_0l}(y))
\overset{(\ref{paja1})}{\leq}
$$
$$
K_2^r\hat{K}\lambda_0^ld(x,y)
=K_2^r\hat{K}\lambda_0^{-\frac{r}{m_0}}\left(\lambda_0^{\frac{1}{m_0}}\right)^nd(x,y)
\leq K_0\lambda^nd(x,y)
$$
proving (\ref{sisi1}).
\end{proof}

Now we prove the following result.

\begin{prop}
 \label{sink}
If $f\in\diff^1(M)$ then every $p\in \Omega(f)$ for which there are $K>0$, $0<\gamma<1$ and $m_0\in\mathbb{N}^+$ satisfying
(\ref{sisi}) is a sink of $f$.
\end{prop}

\begin{proof}
By Lemma \ref{blabla}
there are $\rho>0$, $K_0>0$ and $0<\lambda<1$ satisfying
(\ref{sisi1}).
As $p\in\Omega(f)$
there are sequences $n_k\to\infty$ and
$y_k\in B\left(p,\frac{\rho}{4}\right)$ such that
$$
d(f^{n_k}(y_k),p)\leq \frac{\rho}{4}\quad\quad\forall k\in\mathbb{N}.
$$
As $n_k\to\infty$ and $0<\lambda<1$ we can fix $k$ large such that
\begin{equation}
\label{siso1}
K_0\lambda^{n_k}<\frac{1}{4}.
\end{equation}
Now, take $x\in M$ with $d(x,p)\leq\frac{\rho}{2}$.
Then, $d(x,y_k)\leq \rho$ and so
$$
d(f^{n_k}(x),p)\leq d(f^{n_k}(x),f^{n_k}(y_k))+d(f^{n_k}(y_k),p)
\overset{(\ref{sisi1})}{\leq}
$$
$$
K_0\lambda^{n_k}d(x,y_k)+\frac{\rho}{4}
\leq
\left(K_0\lambda^{n_k}\cdot 2+\frac{1}{2}\right)\frac{\rho}{2}
\overset{(\ref{siso1})}{<}\frac{\rho}{2}.
$$
This together with (\ref{sisi1}) yields
$$
f^{n_k}\left(B\left[p,\frac{\rho}{2}\right]\right)\subset B\left(p,\frac{\rho}{2}\right)
\mbox{ and }
d(f^{n_k}(x),f^{n_k}(y))\leq \beta d(x,y),
$$
$\forall x,y\in B\left[p,\frac{\rho}{2}\right]$
where $\beta=K_0\lambda^{n_k}$ and $B[\cdot,\cdot]$ denotes closed ball operation.
Since $0<\beta<\frac{1}{4}$ by (\ref{siso1}) the contracting map principle provides a fixed point
$x_*\in B\left(p,\frac{\rho}{2}\right)$ of $f^{n_k}$
such that
$$
d((f^{n_k})^l(x),x_*)\leq \beta^ld(x,x_*),
\quad\quad\forall x\in B\left[p,\frac{\rho}{2}\right], l\in\mathbb{N}^+.
$$
In particular, $x_*\in\per(f)$.

As $f$ is Lipschitz continuous there is
$K_2>0$ such that
$d(f(x),f(y))\leq K_2d(x,y)$ for all $x,y\in M$. Take
$$
K_3=\sup\{K_2^r\beta^{-\frac{r}{n_k}}:0\leq r\leq n_k-1\}.
$$
Fix $x\in B\left[p,\frac{\rho}{2}\right]$ and $n\in\mathbb{N}$ large.
Then, $n=ln_k+r$ for some integers $l\geq 1$ and $0\leq r\leq n_k-1$ so
$$
d(f^n(x),f^n(x_*))=d(f^r(f^{ln_k}(x)),f^r(f^{ln_k}(x_*))\leq K_1^rd((f^{n_k})^l(x),x_*)\leq
$$
$$
K_2^r\beta^ld(x,x_*)\leq K_3\beta^{\frac{n}{n_k}}d(x,x_*).
$$
As $0<\beta^{\frac{1}{n_k}}<1$ we get
$$
\lim_{n\to \infty}d(f^n(x),f^n(x_*))\to0,
\quad\quad\forall x\in B\left[p,\frac{\rho}{2}\right].
$$
From this we obtain that $x_*$ is isolated in $\Omega(f)$.
However, $p\in \Omega(f)$ which is invariant so $(f^{n_k})^l(p)\in\Omega(f)$.
Since $d((f^{n_k})^l(p),x_*)\to 0$ as $l\to\infty$ and $x_*$ is isolated in $\Omega(f)$ which is closed
we obtain
$(f^{n_k})^l(p)=x_*$ for some $l\in\mathbb{N}^+$. As $x_*\in \per(f)$ we get
$p\in \per(f)$ so the result from Lemma \ref{short1}.
\end{proof}

We will need the following variation of the so-called {\em Pliss's lemma}
(c.f. Lemma 3.0.2 in \cite{PS}).

\begin{lemma}
\label{pliss}
For every $g\in \diff^1(M)$ and $0<\gamma_1<\gamma_2$ there are
$m\in \mathbb{N}^+$ and $c>0$ such that if $x\in M$ and $n\geq m$ is an integer
satisfying
$$
\prod_{i=1}^n\|Dg(g^i(x))\|\leq \gamma_1^n,
$$
then there are $0\leq n_1<n_2<\cdots<n_l\leq n$ with $l\geq cn$ such that
$$
\prod_{i=n_r}^j\|Dg(g^i(x))\|\leq \gamma_2^{j-n_r},
\quad \mbox{ for every } r\in\{1,\cdots,l\} \mbox{ and } j\in\{ n_r,\cdots,n\}.
$$
\end{lemma}

Next we introduce an auxiliary definition
motivated by R. Ma\~n\'e and V. A. Pliss \cite{M}, \cite{M1},\cite{Pl}.
Denote by $[r]$ the real part of $r\in\mathbb{R}$.

\begin{defi}
Given $f\in\diff^1(M)$ we say that $S\subset M$
is {\em (MP)-contracting with respect to $n:S\to\mathbb{N}^+$}
if there are $K_1>0$, $0<\lambda<1$ and $m_0\in\mathbb{N}^+$ such that
\begin{equation}
\label{principal3}
\prod_{j=0}^{\left[\frac{n(p)}{m_0}\right]-1}\|Df^{m_0}(f^{m_0j}(p))\|
\leq K_1\lambda^{\left[\frac{n(p)}{m_0}\right]},
\quad\forall p\in S\mbox{ with } n(p)\geq m_0.
\end{equation}
\end{defi}

Notice that {\em every} set $S$ is (MP)-contracting with respect to
any bounded map $n:S\to \mathbb{N}^+$.
To see it just take any upper bound $n_0$ of $n$,
any $0<\lambda<1$, $K_1>0$ and $m_0=1+n_0$
to obtain (\ref{principal3}).
From this we deduce that every finite set $S$ is (MP)-contracting with respect to
some map with finite preimages $n: S\to \mathbb{N}^+$
(e.g. the constant map $n(p)=1$).
The following lemma improving
Lemma 2.6 in \cite{m'} proves the converse of the above property in the case we are interested in, namely,
for nonwandering points.

\begin{lemma}
\label{mano}
If $f\in\diff^1(M)$ every set $S\subset\Omega(f)$ which is (MP)-contracting
with respect to some map with finite preimages $n:S\to \mathbb{N}^+$ is finite.
\end{lemma}

\begin{proof}
Suppose by contradiction that $S$ is infinite.
Fix $K_1>0$, $0<\lambda<1$ and $m_0\in\mathbb{N}^+$
such that (\ref{principal3}) holds for $n$.
Since $S$ is infinite there is an infinite sequence $p_k\in S$ and,
since $n$ has finite preimages,
we can assume up to passing to a subsequence if necessary that $n(p_k)\to\infty$.
Choose constants $0<\lambda<\gamma_1<\gamma_2<1$ and let
$m\in\mathbb{N}^+$ and $c>0$ be as in the Pliss's lemma for these choices.
As $n(p_k)\to\infty$ we have $\left[\frac{n(p_k)}{m_0}\right]\to\infty$ too.
In particular we can assume $\left[\frac{n(p_k)}{m_0}\right]\geq m$ and,
since $\lambda<\gamma_1$, we can also assume that
$K_1\lambda^{\left[\frac{n(p_k)}{m_0}\right]}\leq \gamma_1^{\left[\frac{n(p_k)}{m_0}\right]}$ for all $k$.
Then, (\ref{principal3}) and the Pliss's lemma imply that for all $k\in\mathbb{N}$ there are
$0\leq n_1^k<n_2^k<\cdots<n^k_{l_k}\leq \left[\frac{n(p_k)}{m_0}\right]-1$ with
$l_k\geq c\left(\left[\frac{n(p_k)}{m_0}\right]-1\right)$ such that
$$
\prod_{j=n^k_r}^s\|Df^{m_0}(f^{m_0j}(p_k))\|\leq \gamma_2^{s-n^k_r},
$$
\begin{equation}
 \label{principal2}
\quad \forall r\in\{1,\cdots,l_k\} \mbox{ and } s\in\left\{ n_r^k,\cdots,
\left[\frac{n(p_k)}{m_0}\right]-1\right\}
\end{equation}
Define the sequence
$x_k=f^{m_0n^k_1}(p_k)$. Since $M$ is compact we can be assume that
$x_k\to p$ for some $p\in M$.

Let us prove that $f$, $p$, $K=\gamma_2^{-1}$, $\gamma=\gamma_2$ and $m_0$ as above
satisfy (\ref{sisi}).

Fix $l\in \mathbb{N}^+$.

Since $l_k\geq c\left(\left[\frac{n(p_k)}{m_0}\right]-1\right)$ we get
$\left(\left[\frac{n(p_k)}{m_0}\right]-1\right)-n^k_1\geq l_k\geq c
\left(\left[\frac{n(p_k)}{m_0}\right]-1\right)$ so
$\left[\frac{n(p_k)}{m_0}\right]-n^k_1\to\infty$ as $k\to\infty$.
Then,
\begin{equation}
\label{ladilla1}
l\leq \left[\frac{n(p_k)}{m_0}\right]-n_1^k,
\quad\quad\mbox{ for }k\mbox{ large enough}.
\end{equation}
On the other hand,
$$
\displaystyle\prod_{i=0}^{l-1}\|Df^{m_0}(f^{m_0i}(x_k))\|
=\displaystyle\prod_{j=n_1^k}^{n_1^k+l-1}\|Df^{m_0}(f^{m_0j}(p_k))\|.
$$
Taking $s=n_1^k+l-1$ one has
$s\in\left\{ n_r^k,\cdots,
\left[\frac{n(p_k)}{m_0}\right]-1\right\}$ by (\ref{ladilla1}) so
$$
\displaystyle\prod_{i=0}^{l-1}\|Df^{m_0}(f^{m_0i}(x_k))\|
\leq \gamma_2^{l-1}=K\gamma^l
$$
by (\ref{principal2}) with $r=1$.

Since $l$ is fixed we can take $k\to\infty$ above to obtain
$$
\displaystyle\prod_{i=0}^{l-1}\|Df^{m_0}(f^{m_0i}(p))\|\leq K\gamma^l.
$$
As $l\in\mathbb{N}^+$ is arbitrary we obtain (\ref{sisi}).

Finally, since $p_k\in \Omega(f)$ and $\Omega(f)$ is invariant one has $x_k\in \Omega(f)$
thus $p\in \Omega(f)$. Then, since (\ref{sisi}) holds, Proposition \ref{sink} implies that $p$ is a sink of $f$.
Since the sinks are isolated in $\Omega(f)$ and $p_k\in \Omega(f)$ we conclude that the sequence
$p_k$ is finite, a contradiction.
This ends the proof.
\end{proof}

The next lemma is essentially contained in \cite{M}.
Given $f\in\diff^1(M)$ we define
$$
m(f)=\inf_{p\in M}m(Df(x)),
$$
where $m(\cdot)$ is the co-norm operation induced by $\|\cdot\|$.

\begin{lemma}
 \label{l1}
For every $f\in\diff^1(M)$, $\epsilon_0>0$, $m\in \mathbb{N}^+$, $n\in\mathbb{N}$ with $n\geq m$
and $p\in M$ there is a sequence of linear isomorphisms
$L_i: T_{f^i(p)}M\to T_{f^{i+1}(p)}M$, $0\leq i\leq n-1$, such that
\begin{enumerate}
 \item[(i)]
$\|L_i-Df(f^i(p))\|\leq \|f\|(2+\epsilon_0)\epsilon_0$ for $0\leq i\leq n-1$.
\item[(ii)]
$$
\displaystyle\prod_{j=0}^{\left[\frac{n}{m}\right]-1}\|Df^m(f^{mj}(p))\|
\leq
(1+\epsilon_0)^{-r}
(m(f))^{-r}\left((1+\epsilon_0)^m\frac{\epsilon_0}{\dim(M)}\right)^{-\left[\frac{n}{m}\right]}
\left\|
\displaystyle\prod_{i=0}^{n-1}L_i
\right\|,
$$
where $r=n-\left[\frac{n}{m}\right]m$.
\end{enumerate}
\end{lemma}

\begin{proof}
We follow closely the proof of Lemma II.5 in \cite{M}.
Fix any nonzero vector $v\in T_pM$. By Lemma II.6 in \cite{M} we can choose linear isomorphisms
$P_j: T_{f^{mj}(p)}M\to T_{f^{mj}(p)}M$, $0\leq j\leq \left[\frac{n}{m}\right]-1$
such that $\|P_j-I\|\leq \epsilon_0$
and
$$
\|(Df^m(f^{mj}(p))P_j)v_j\|\geq \left(\frac{\epsilon_0}{\dim(M)}\right)
\|Df^m(f^{mj}(p))\|\cdot\|v_j\|,
\quad\quad
\forall 0\leq j\leq \left[\frac{n}{m}\right]-1,
$$
where $v_j$ is defined inductively by
$v_0=v$ and
$$
v_j=Df^m(f^{m(j-1)}(p))P_{j-1}v_{j-1},
\quad\quad \forall j\geq 1.
$$
This allows us to define the sequence $L_i: T_{f^i(p)}M\to T_{f^{i+1}(p)}M$,
$0\leq i\leq n-1$, by
$$
L_i = \left\{
\begin{array}{rcl}
(1+\epsilon_0)Df(f^{mj}(p)) P_j &\mbox{if} &i=mj\mbox{ is a multiple of }m,\\
(1+\epsilon_0)Df(f^i(p))&  &\mbox{otherwise}.
\end{array}
\right.
$$
Let us prove that this sequence works.

First observe that
$\|L_i-Df(f^i(p))\|$ is either
$$
\|(1+\epsilon_0)Df(f^{mj}(p))P_j-Df(f^{mj}(p))\|
\quad \mbox{ or }
\quad
\|(1+\epsilon_0)Df(f^i(p))-Df(f^i(p))\|
$$
depending on whether
$i=mj$ is a multiple of $m$ or not.
In the former case we have
$$
\|(1+\epsilon_0)Df(f^i(p))P_j-Df(f^i(p))\|
\leq
$$
$$
\|(1+\epsilon_0)Df(f^{mj}(p))P_j-(1+\epsilon_0)Df(f^{mj}(p))\|
+
\|(1+\epsilon_0)Df(f^{mj}(p))-Df(f^{mj}(p))\|
$$
$$
\leq\|Df(f^{mj}(p))\|\cdot (1+\epsilon_0)\cdot\|P_j-I\|
+
\|Df(f^{mj}(p))\|\epsilon_0\leq
$$
$$
\|f\|(1+\epsilon_0)\epsilon_0+\|f\|\epsilon_0
=\|f\|(2+\epsilon_0)\epsilon_0
$$
and, in the later,
$$
\|L_i-Df(f^i(p))\|=\|(1+\epsilon_0)Df(f^i(p))-Df(f^i(p))\|\leq \|f\|\epsilon_0
\leq \|f\|(2+\epsilon_0)\epsilon_0
$$
proving (i).

To prove (ii) we notice that
$$
\left\|
\displaystyle\prod_{i=0}^{n-1}L_i
\right\|
\cdot \|v\|
\geq
\left\|
\left(
\displaystyle\prod_{i=0}^{n-1}L_i
\right)
v
\right\|
=
(1+\epsilon_0)^{n}
\left\|
\left(
\displaystyle\prod_{i=0}^{n-1}\hat{L}_i
\right)
v
\right\|,
$$
where $\hat{L}_i$ is either $Df(f^{mj}(p))P_j$ or $Df(f^i(p))$
depending on whether
$i=mj$ is a multiple of $m$ or not.
Thus, the choice of the $P_j$'s implies
$$
\left\|
\displaystyle\prod_{i=0}^{n-1}L_i
\right\|
\cdot \|v\|
\geq
(1+\epsilon_0)^{n}(m(f))^r
\left\|
\left(
\displaystyle\prod_{j=0}^{\left[\frac{n}{m}\right]-1}
Df^m(f^{mj}(p))P_j
\right)
v
\right\|
\geq
$$
$$
(1+\epsilon_0)^{n}(m(f))^r
\left(
\frac{\epsilon_0}{\dim(M)}
\right)^{\left[\frac{n}{m}\right]}
\cdot
\left(
\displaystyle\prod_{j=0}^{\left[\frac{n}{m}\right]-1}
\| Df^m(f^{mj}(p))\|
\right)
\cdot
\|v\|
$$
hence
$$
\displaystyle\prod_{j=0}^{\left[\frac{n}{m}\right]-1}
\|Df^m(f^{mj}(p))\|
\leq
(1+\epsilon_0)^{-n}(m(f))^{-r}
\left(
\frac{\epsilon_0}{\dim(M)}
\right)^{-\left[\frac{n}{m}\right]}
\left\|
\displaystyle\prod_{i=0}^{n-1}
L_i
\right\|.
$$
But the definition of $r$ implies
$$
(1+\epsilon_0)^{-n}
\left(
\frac{\epsilon_0}{\dim(M)}
\right)^{-\left[\frac{n_p}{m}\right]}
=\left(
(1+\epsilon_0)^m\frac{\epsilon_0}{\dim(M)}
\right)^{-\left[\frac{n}{m}\right]}(1+\epsilon_0)^{-r}
$$
yielding (ii).
\end{proof}

From this lemma we obtain the following corollary.

\begin{clly}
\label{laprueva1}
Every set which is uniformly bounded with respect
to $n: S\to \mathbb{N}^+$ is also (MP)-contracting with respect to $n$.
\end{clly}

\begin{proof}
Let $S$ be uniformly bounded with respect to $n:S\to \mathbb{N}^+$, i.e.,
there are positive constants $\epsilon$ and $K$ satisfying Definition \ref{def1}.
Fix $\epsilon_0>0$ with
$\|f\|(2+\epsilon_0)\epsilon_0<\epsilon$ and, for such an $\epsilon_0$, take
$m_0\in\mathbb{N}^+$ such that
$(1+\epsilon_0)^{m_0}\frac{\epsilon_0}{\dim(M)}>1$.
Define
$$
K_1=K\sup_{0\leq r\leq m_0-1}(1+\epsilon_0)^{-r}(m(f))^{-r}
\quad\mbox{ and }\quad
\lambda=\left((1+\epsilon_0)^{m_0}\frac{\epsilon_0}{\dim(M)}\right)^{-1}.
$$
Notice that $K_1>0$ and $0<\lambda<1$.

Now suppose that $p\in S$ satisfies $n(p)\geq m_0$.
Then, applying Lemma \ref{l1} to $n=n(p)$ and $m=m_0$,
we obtain a sequence of linear isomorphisms
$L_i: T_{f^i(p)}M\to T_{f^{i+1}(p)}M$, $0\leq i\leq n(p)-1$ satisfying
(i) and (ii) of this lemma.
By (i) and the choice of $\epsilon_0$ above we obtain
$\|L_i-Df(f^i(p))\|\leq\epsilon$ for $0\leq i\leq n(p)-1$. Then,
$$
\left\|\displaystyle\prod_{i=0}^{n(p)-1}L_i\right\|\leq K
$$
by Definition \ref{def1}.
Replacing in (ii) we get
$$
\displaystyle\prod_{j=0}^{\left[\frac{n(p)}{m_0}\right]-1}\|Df^{m_0}(f^{m_0j}(p))\|
\leq
(1+\epsilon_0)^{-r}
(m(f))^{-r}\left((1+\epsilon_0)^{m_0}\frac{\epsilon_0}{\dim(M)}\right)^{-\left[\frac{n(p)}{m_0}\right]}
K,
$$
where $r=n(p)-\left[\frac{n(p)}{m_0}\right]m_0$.
But clearly $0\leq r\leq m_0-1$ so the definition of $K_1$ and $\lambda$ above yields (\ref{principal3}).
This proves the result.
\end{proof}

\begin{proof}[Proof of Theorem \ref{thA}]
If $S$ is uniformly bounded
with respect to some map with finite preimages $n:S\to \mathbb{N}^+$, then
it is (MP)-contracting with respect to $n$ by Corollary \ref{laprueva1}.
If, additionally, $S\subset \Omega(f)$ and $n$ has finite preimages, then $S$ is finite by Lemma \ref{mano}.
\end{proof}

\section{Applications}

\noindent
In this section we give applications of Theorem \ref{thA} by proving
corollaries \ref{cofin}, \ref{necessary1}, \ref{M-S}, \ref{li-pli} and \ref{finite-sinks}.
The key ingredient is the lemma below
\cite{F} (c.f. Lemma 2.1.1 in \cite{PS}).

\begin{lemma}
\label{frank}[Franks's lemma]
Let $f\in \diff^1(M)$ and $\mathcal{W}(f)\subset \diff^1(M)$ be a neighborhood
of $f$.
Then, there are $\epsilon>0$ and a neighborhood $\mathcal{ W}_0(f)\subset
\mathcal{W}(f)$ of $f$ such that if $g'\in \mathcal{ W}_0(f)$, $\{x_1,\cdots ,x_n\}\subset M$
is a finite set, $U\subset M$ is a neighborhood of $\{x_1,\cdots ,x_n\}$
and $L_i:T_{x_i}M\to T_{g'(x_i)}M$ are linear maps satisfying
$\| L_i-Dg'(x_i)\| <\epsilon$ ($\forall i=1,\cdots ,n$),
then there is $g\in \mathcal{ W}(f)$ such that $g(x)=g'(x)$
in $\{x_1,\cdots ,x_n\}\cup (M\setminus U)$, and
$Dg(x_i)=L_i$ for every $i=1,\cdots ,n$.
\end{lemma}

A first application of this lemma is the following.

\begin{lemma}
 \label{proof-cofin}
Let $f\in\diff^1(M)$, $S\subset M$ and $n: S\to\mathbb{N}^+$ be such that
there is a neighborhood $\mathcal{U}$ of $f$ satisfying
$$
\sup_{(h,p)\in \mathcal{U}\times S}\|Dh^{n(p)}(p)\|<\infty.
$$
Then, $S$ is uniformly bounded with respect to $n$.
\end{lemma}

\begin{proof}
Take $\mathcal{W}(f)=\mathcal{U}$ in the Franks's lemma to obtain
the neighborhood $\mathcal{W}_0(f)\subset\mathcal{W}(f)$
and $\epsilon>0$.
Take
$$
K=\sup_{(h,p)\in \mathcal{U}\times S}\|Dh^{n(p)}(p)\|
$$
Clearly $K>0$.
Let $p\in S$ and $L_i: T_{f^i(p)}M\to T_{f^{i+1}(p)}M$ be a sequence of linear isomorphisms
with $\|L_i-Df(f^i(p))\|\leq \epsilon$.
Put $g'=f$ and $x_i=f^i(p)$ for $0\leq i\leq n(p)-1$.
Clearly $f\in\mathcal{W}_0(f)$ so, by the Franks's lemma,
there is $g\in\mathcal{W}(f)=\mathcal{U}$ with
$g(x)=f(x)$
in $\{x_0,\cdots ,x_{n(p)-1}\}$ and
$Dg(x_i)=L_i$ for every $i=0,\cdots ,n(p)-1$.
It follows that $g^i(p)=x_i$ for $i=0,\cdots ,n(p)-1$
thus
$$
Dg^{n(p)}(p)=\prod_{i=0}^{n(p)-1}Dg(g^i(p))=\prod_{i=0}^{n(p)-1}Dg(x_i)
=\prod_{i=0}^{n(p)-1}L_i,
$$
whence
$$
\left\|\displaystyle\prod_{i=0}^{n(p)-1}L_i\right\|=\|Dg^{n(p)}(p)\|\leq
\sup_{(h,p)\in \mathcal{U}\times S}\|Dh^{n(p)}(p)\|=K
$$
proving the result.
\end{proof}

\begin{proof}[Proof of Corollary \ref{cofin}]
Apply Theorem \ref{thA} and Lemma \ref{proof-cofin}.
\end{proof}

Next we observe that if $f\in\diff^1(M)$ every $S\subset\per(f)$ comes equipped
with a natural map $p\in S\mapsto n_{p,f}\in \mathbb{N}^+$ which will be referred to as
the {\em period map}. This motivates the following auxiliary definition.

\begin{defi}
A set of periodic points is
{\em uniformly bounded at the period} if it is uniformly bounded
with respect to its period map.
\end{defi}

Notice that if $f$ has finitely many periodic points of period
$n$, for all $n\in\mathbb{N}^+$, then all period maps have finite preimages.
From this, Theorem \ref{thA} and the fact that all periodic points are nonwandering
we obtain the following useful corollary:

\begin{clly}
\label{useful}
Let $f$ be a diffeomorphisms of a closed manifold
having finitely many periodic points of period $n$, $\forall n\in\mathbb{N}^+$.
Then, every set of periodic points of $f$ which is uniformly bounded at the period
is finite.
\end{clly}

This corollary motivates the search of sufficient conditions for a given
set of periodic points to be uniformly bounded at the period.
In the sequel we put a condition based on the following
definition inspired by \cite{M}.
Given a linear operator $A:V\to V$ we denote by
$$
\spec(A)=\max\{|\lambda|:\lambda \mbox{ is an eigenvalue of }
A\}
$$
its the spectral radius.

\begin{defi}
 \label{def2}
We say that $S\subset \per(f)$  is
{\em spectrally uniformly bounded at the period} if there are positive constants $\epsilon,B$ such that
$$
\spec\left(\prod_{i=0}^{n_p-1}L_i\right)<B
$$
for all $p\in S$ and all sequence
of linear maps $L_i: T_{f^i(p)}M\to T_{f^{i+1}(p)}M$ with
$\|L_i-Df(f^i(p))\|\leq \epsilon$ for $0\leq i\leq n_p-1$.
\end{defi}

Since $\spec(A)\leq \|A\|$
one has that every set of periodic points which is uniformly bounded at the period is also spectrally uniformly
bounded at the period.
The converse is true by the following lemma.

\begin{lemma}
\label{c1}
Every set of periodic points which is spectrally uniformly bounded at the period is
uniformly bounded at the period.
\end{lemma}

\begin{proof}
The proof follows as in Lemma II.4 of \cite{M}. We include the details for the sake of completeness.
Denote by $I$ the identity and by $\tr(\cdot)$ the trace operation.

Let $S$ be a subset of periodic points of $f$ which is spectrally uniformly bounded at the period.
Fix $B$ and $\epsilon$ as in Definition \ref{def2} for $S$. Fix also $K>0$ such that
$$
\frac{4(\epsilon+\|f\|)\dim(M)\cdot B}{K}\leq \frac{\epsilon}{2}.
$$
Suppose by contradiction that $S$ is not uniformly bounded at the period.
Then, there are $p\in S$
and linear mappings $L_i: T_{f^i(p)}M\to T_{f^{i+1}(p)}M$ with
$\|L_i-Df(f^i(p))\|\leq \frac{\epsilon}{2}$ for $0\leq i\leq n_p-1$
such that $\prod_{i=0}^{n_p-1}L_i$ has an entry $b$ with $|b|\geq K$.
Taking a suitable $\dim(M)\times \dim(M)$ matrix $A$ with all entries $0$ except one with value
$\frac{4\dim(M)\cdot B}{K}$ we get
$$
\tr\left[\left(
\displaystyle\prod_{i=0}^{n_p-1}L_i\right)\cdot (I+A)\right]
=
b\cdot\frac{4\dim(M)\cdot B}{K}+\tr\left(\displaystyle\prod_{i=0}^{n_p-1}L_i\right).
$$
But Definition \ref{def2} says that
$\spec\left(\prod_{i=0}^{n_p-1}L_i\right)<B$, so,
$$
\left|
\tr\left(
\displaystyle\prod_{i=0}^{n_p-1}L_i\right)
\right|\leq \dim(M)\cdot B
$$
thus
$$
\left|
\tr\left[\left(
\displaystyle\prod_{i=0}^{n_p-1}L_i\right)\cdot(I+A)\right]
\right|\geq |b|\cdot\frac{4\dim(M)\cdot B}{K}- \dim(M)\cdot B
\geq 3\dim(M)\cdot B.
$$
From this we get that
$\left(\prod_{i=0}^{n_p-1}L_i\right)\cdot (I+A)$ has spectral radius bigger than $B$.

Next we define the new sequence of linear maps
$\hat{L}_i: T_{f^i(p)}M\to T_{f^{i+1}(p)}M$
by $\hat{L}_0=L_0(I+A)$ and $\hat{L}_i=L_i$ for $1\leq i\leq n_p-1$.

On the one hand,
$$
\displaystyle\prod_{i=0}^{n_p-1}\hat{L}_i=\left(\displaystyle\prod_{i=0}^{n_p-1}L_i\right)\cdot (I+A)
$$
and so
$\prod_{i=0}^{n_p-1}\hat{L}_i$ has spectral radius bigger than $B$.

On the other hand,
$\|\hat{L}_i-Df(f^i(p))\|\leq \frac{\epsilon}{2}\leq \epsilon$ for $1\leq i\leq n_p-1$ and
$$
\|\hat{L}_0-Df(p)\|=\|L_0(I+A)-L_0\|+\|L_0-Df(p)\|
\leq
$$
$$
\|L_0\|\cdot \frac{4\dim(M)\cdot B}{K}+\frac{\epsilon}{2}
\leq
\frac{4(\epsilon+\|f\|)\dim(M)\cdot B}{K}+\frac{\epsilon}{2}
\leq \frac{\epsilon}{2}+\frac{\epsilon}{2}=\epsilon.
$$
So, by Definition \ref{def2},
$\prod_{i=0}^{n_p-1}\hat{L}_i$ has spectral radius less than or equal to $B$.
This is clearly a contradiction which proves the result.
\end{proof}

\begin{proof}[Proof of Corollary \ref{necessary1}]
Clearly $\per^*(f)\neq\emptyset$ if and only if $\per(f)\neq\sink(f)$.
Then, to prove the result, we only need to prove that
$\{f\in\diff^1(M):\per(f)\neq\sink(f)\}$ is dense in $\diff^1(M)$.
Suppose by contradiction that it is not so. Then,
there are $f\in\diff^1(M)$ and a neighborhood $\mathcal{U}$ of it
such that $\per(h)=\sink(h)$ for all $h\in\mathcal{U}$.
Take $\mathcal{W}(f)=\mathcal{U}$ in the Franks's lemma to obtain $\epsilon$ and $\mathcal{W}_0(f)\subset
\mathcal{W}(f)$.

Suppose for a while that there is $g'\in \mathcal{W}_0(f)$ such that
$\sink(g')$ is not uniformly bounded at the period.
Then, $\sink(g')$ is not spectrally uniformly bounded at the period by Lemma \ref{c1}.
It follows that
there is $p\in \sink(g')$ and a sequence of linear isomorphisms
$L_i: T_{x_i}M\to T_{g'(x_i)}M$ with $x_i=(g')^i(p)$
such that $\|L_i-Dg'(x_i)\|\leq \epsilon$ for all $0\leq i\leq n_{p,g'}-1$
and $\spec\left(\prod_{i=0}^{n_{p,g'}-1}L_i\right)>1$.
Then, by the Franks's lemma,
then there is $g\in \mathcal{U}$ such that $g(x)=g'(x)$
in $\{x_0,\cdots ,x_{n_{p,g'}-1}\}$ and
$Dg(x_i)=L_i$ for every $i=0,\cdots ,n_{p,g'}-1$.
Then, $p\in \per(g)$, $n_{p,g}=n_{p, g'}$ and $Dg^{n_{p,g}}(p)=\prod_{i=0}^{n_{p,g'}-1} L_i$.
So, $\spec(Dg^{n_{p,g}}(p))=\spec\left(\prod_{i=0}^{n_{p,g'}-1}L_i\right)>1$
thus $p\not\in\sink(g)$ whence $\per(g)\neq\sink(g)$ contradicting $g\in \mathcal{U}$.
Therefore, $\sink(g')$ is uniformly bounded at the period
for all $g'\in \mathcal{W}_0(f)$.

Now, by the Kupka-Smale and general density theorems \cite{pt}, \cite{p}, there is a residual subset
$\mathcal{U}'\subset \mathcal{W}_0(f)$ such that
every periodic point of $g' \in\mathcal{U}'$
is hyperbolic (thus with finitely many ones of period
$n$, for all $n\in\mathbb{N}^+$), the invariant manifolds of these periodic points are in general position
and $\per(g')$ is dense in $\Omega(g')$.
Since we also have that $\sink(g')$ is uniformly bounded at the period
we conclude that $\sink(g')$ is finite, by Corollary \ref{useful}, whence $\per(g')$ is finite for $g'\in\mathcal{U}'$.
We conclude that every $g'\in\mathcal{U}'$ is Morse-Smale.
As every Morse-Smale diffeomorphism has at least one source we obtain a contradiction
which proves the result.
\end{proof}

To prove Corollary \ref{M-S} we need the following lemma.

\begin{lemma}
 \label{appli1}
If $f\in\diff^1(M)$ exhibits a neighborhood $\mathcal{U}$
satisfying (\ref{ms1}), then $\per(f)$ is uniformly bounded at the period.
\end{lemma}

\begin{proof}
First of all notice that (\ref{ms1}) implies
$$
\sup_{(h,q)\in\mathcal{U}\times\per^*(h)} \spec\left(Dh^{n_{q,h}}(q)\right)<\infty.
$$
Take $\mathcal{W}(f)=\mathcal{U}$ in the Franks's lemma to obtain
$\epsilon$ and $\mathcal{W}_0(f)\subset \mathcal{U}$.

Suppose by contradiction that $\per(f)$ is not uniformly bounded at the period.
Then, $\per(f)$ is not spectrally uniformly bounded at the period by Lemma \ref{c1}. It follows that
there exist $p\in \per(f)$ and a sequence $L_i:T_{f^i(p)}M\to T_{f^{i+1}(p)}M$
such that $\|L_i-Df(f^i(p))\|\leq \epsilon$ for $0\leq i\leq n_p-1$ but
\begin{equation}
\label{ms2}
\spec\left(\displaystyle\prod_{i=0}^{n_p-1} L_i\right)>
\max\left\{1,\sup_{(h,q)\in\mathcal{U}\times\per^*(h)} \spec\left(Dh^{n_{q,h}}(q)\right)\right\}.
\end{equation}
Take $g'=f$ and $x_i=f^i(p)$ for $0\leq i\leq n_p-1$.
Clearly $g'\in\mathcal{W}_0(f)$ and so, by the Franks's lemma,
there is $g\in \mathcal{U}$ with $g=f$ in $\{x_0,\cdots, x_{n_p-1}\}$
such that $Dg(g^i(p))=L_i$.
Then, $p\in \per(g)$, $n_{p,g}=n_p$ and $Dg^{n_{p,g}}(p)=\displaystyle\prod_{i=0}^{n_p-1} L_i$.
Replacing in (\ref{ms2}) we get
\begin{equation}
\label{ms1-spec}
\spec(Dg^{n_{p,g}}(p))> \max\left\{1,\sup_{(h,q)\in\mathcal{U}\times\per^*(h)} \spec\left(Dh^{n_{q,h}}(q)\right)\right\}.
\end{equation}
In particular, $\spec(Dg^{n_{p,g}}(p))>1$ thus $p\in \per^*(g)$.
As $g\in \mathcal{U}$ we get
$$
\spec(Dg^{n_{p,g}}(p))\leq \sup_{(h,q)\in\mathcal{U}\times\per^*(h)} \spec\left(Dh^{n_{q,h}}(q)\right)
$$
contradicting (\ref{ms1-spec}).
\end{proof}

\begin{proof}[Proof of Corollary \ref{M-S}]
Let $f$ be a diffeomorphism of a closed manifold.
If there is a neighborhood $\mathcal{U}$ of $f$ satisfying
(\ref{ms1}), then $\per(f)$ is uniformly
bounded at the period by Lemma \ref{appli1}. If, additionally, $f$ has
finitely many periodic points of period $n$, $\forall n\in\mathbb{N}^+$, then
$\per(f)$ is finite by Corollary \ref{useful}.
\end{proof}

To prove Corollary \ref{li-pli} we will need two short lemmas and the following definition
inspired in \cite{Pl}.

\begin{defi}
\label{coco}
If $f\in\diff^1(M)$ we say that $S\subset \per(f)$ is
{\em uniformly coarse} if there is
$\epsilon>0$ such that
$\prod_{i=0}^{n_p-1}L_i$ has no eigenvalues of modulus $1$
for all $p\in S$ and all sequence of linear isomorphisms
$L_i:T_{f^i(p)}M\to T_{f^{i+1}(p)}M$
with $\|L_i-Df(f^i(p))\|\leq \epsilon$ for all integer
$0\leq i\leq n_p-1$.
\end{defi}

Related to this definition we have the following lemma.

\begin{lemma}
\label{coarse1}
If $f\in\diff^1(M)$ every uniformly coarse set $S\subset \sink(f)$
is uniformly bounded at the period.
\end{lemma}

\begin{proof}
Take $\epsilon$ as in Definition \ref{coco} for $S$ and
suppose by contradiction that $S$ is not uniformly bounded at the period.
Then, $S$ is
not spectrally uniformly bounded at the period by Lemma \ref{c1}.
Thus, there are $p\in S$ and a sequence of linear isomorphisms
$L_i:T_{f^i(p)}M\to T_{f^{i+1}(p)}M$
with $\|L_i-Df(f^i(p))\|\leq \epsilon$ for all integer
$0\leq i\leq n_p-1$ such that
$\spec\left(\prod_{i=0}^{n_p-1}L_i\right)>1$.
Take a path of linear isomorphisms $L_i^t:T_{f^i(p)}M\to T_{f^{i+1}(p)}M$ for $0\leq t\leq 1$ and
$0\leq i\leq n_p-1$ such that $L_i^0=Df(f^i(p))$, $L^1_i=L_i$
and $\|L^t_i-Df(f^i(p))\|\leq \epsilon$ for all $0\leq t\leq 1$ and
$0\leq i\leq n_p-1$.
Since $p\in\sink(f)$ (thus $\spec(Df^{n_p}(p))<1$) there is
$0< t< 1$ such that
$\prod_{i=0}^{n_p-1}L^t_i$ has an eigenvalue of modulus $1$. But this is a contradiction
since $S$ is uniformly coarse.
\end{proof}

In particular, we obtain the following corollary.

\begin{clly}
\label{coro1}
Let $f$ be a diffeomorphisms of a closed manifold
having finitely many periodic points of period $n$, $\forall n\in\mathbb{N}^+$.
Then, every uniformly coarse set of $f$ in $\sink(f)$ is finite.
\end{clly}

\begin{proof}
If $S\subset \sink(f)$ is uniformly coarse, then it is uniformly bounded at the period by Lemma \ref{coarse1}.
Therefore, it is finite by Corollary \ref{useful}.
\end{proof}

Another lemma is as follows.

\begin{lemma}
\label{coarse2}
If $f\in\diff^1(M)$ is a star diffeomorphisms, then $\per(f)$ is uniformly coarse.
\end{lemma}

\begin{proof}
Suppose by contradiction that $\per(f)$ is not uniformly coarse.
Since $f$ is star there is a neighborhood $\mathcal{W}(f)$ of $f$
such that every periodic point of every $g\in \mathcal{W}(f)$ is hyperbolic.
For such a neighborhood we take $\mathcal{W}_0(f)$ and $\epsilon$ as in the Frank's lemma.
Since $\per(f)$ is not uniformly coarse
there are $p\in\per(f)$ and a sequence of linear isomorphisms
$L_i: T_{f^i(p)}M\to T_{f^{i+1}(p)}M$
with $\|L_i-Df(f^i(p))\|\leq\epsilon$ for $0\leq i\leq n_p-1$
such that $\prod_{i=0}^{n_p-1}L_i$ has an eigenvalue of modulus $1$.
Take $x_i=f^i(p)$, $0\leq i\leq n_p-1$, and $g'=f$.
Clearly $g'\in \mathcal{W}_0(f)$ thus, by the Franks's lemma,
there is $g\in \mathcal{W}(f)$ with $g=f$ in $\{x_0,\cdots, x_{n_p-1}\}$
such that $Dg(x_i)=L_i$ for all $0\leq i\leq n_p-1$.
It turns out that $p\in\per(g)$, $n_{p,g}=n_p$ and
$Dg^{n_{p,g}}(p)=\prod_{i=0}^{n_p-1}L_i$.
It follows that $p$ is a nonhyperbolic periodic point of $g$ contradicting $g\in\mathcal{W}(f)$.
This ends the proof.
\end{proof}

\begin{proof}[Proof of Corollary \ref{li-pli}]
If $f\in\diff^1(M)$ is star
$\per(f)$ (and so $\sink(f)$) are uniformly coarse
by Lemma \ref{coarse2}.
Since all star diffeomorphisms on compact manifolds
have finitely many periodic points of period $n$, $\forall n$, we conclude that $\sink(f)$
is finite by Corollary \ref{coro1}.
Since the set of sources of $f$ is $\sink(f^{-1})$ and $f^{-1}$ is star (for $f$ is) we obtain that
$f$ has finitely many sources too.
\end{proof}

\begin{proof}[Proof of Corollary \ref{finite-sinks}]
Let $f$ be a $C^1$ generic diffeomorphism of a closed manifold $M$ satisfying
$$
\cl(\sad(f)\cup\sou(f))\cap \cl(\sink(f))=\emptyset.
$$
Then, there is a neighborhood $U$
of $\cl(\sad(f)\cup\sou(f))$ such that
$
U\cap \sink(f)=\emptyset.
$
But $f$ is $C^1$ generic so we can assume that the set-valued map
$h\in\diff^1(M)\mapsto \cl(\sad(h)\cup\sou(h))$ is semicontinous at $f$.
From this we obtain a neighborhood $\mathcal{U}$
of $f$ such that
$
\cl(\sad(h)\cup\sou(h))\subset U
$, $\forall h\in \mathcal{U}$.
We conclude that
\begin{equation}
 \label{lele}
(\sad(h)\cup\sou(h))\cap \sink(f)=\emptyset,
\quad\quad\forall h\in \mathcal{U}.
\end{equation}
We can further assume by the Kupka-Smale theorem \cite{pt}
that every periodic point of $f$ is hyperbolic, so,
there are finitely many ones of period $n$, $\forall n\in\mathbb{N}^+$.

Let us prove that $\sink(f)$ is uniformly coarse.
Indeed, suppose by contradiction that it is not so.
Take $\mathcal{W}(f)=\mathcal{U}$ in the Franks's lemma to obtain
$\mathcal{W}_0(f)$ and $\epsilon$.
Since $\sink(f)$ is not uniformly coarse
we can find $p\in \sink(f)$ together with linear mappings
$L_i: T_{f^i(p)}M\to T_{f^{i+1}(p)}M$
with $\|L_i-Df(f^i(p))\|\leq \epsilon$ for $0\leq i\leq n_p-1$ such that
$\prod_{i=0}^{n_p-1}L_i$ has an eigenvalue of modulus $1$.
Put $g'=f$ and $x_i=f^i(p)$ for $0\leq i\leq n(p)-1$.
Evidently $f\in\mathcal{W}_0(f)$ so, by the Franks's lemma,
there is $g\in\mathcal{W}(f)=\mathcal{U}$ such that
$g(x)=f(x)$
in $\{x_0,\cdots ,x_{n_p-1}\}$ and
$Dg(x_i)=L_i$ for every $i=0,\cdots ,n_p-1$.
It follows that $g^i(p)=x_i$ for $i=0,\cdots ,n_p-1$
thus $p\in\per(g)$, $n_{p,g}=n_p$ and
$Dg^{n(p)}(p)=\prod_{i=0}^{n(p)-1}L_i$.
Since $\prod_{i=0}^{n_p-1}L_i$ has an eigenvalue of modulus $1$
we conclude that $p$ (as a periodic point of $g$) has also an eigenvalue of modulus $1$.
Then, we can assume that
$p\in \sad(g)\cup\sou(g)$ by perturbing a $g$ a bit if necessary.
But $p\in \sink(f)$ so $p\in (\sad(g)\cup\sou(g))\cap \sink(f)$ whence
$(\sad(g)\cup\sou(g))\cap \sink(f)\neq\emptyset$.
Since $g\in\mathcal{U}$ we contradict (\ref{lele}) thus
$\sink(f)$ is uniformly coarse.

Since $f$ has finitely many periodic points of period $n$, $\forall n\in\mathbb{N}^+$,
we conclude that $\sink(f)$ is finite by Corollary \ref{coro1}.
\end{proof}

\end{document}